\documentclass{article}
\usepackage{graphicx} % Required for inserting images
\usepackage{amssymb}
\usepackage{amsmath}
\usepackage{amsthm}
\usepackage{indentfirst}
\newtheorem{theorem}{Theorem}
\usepackage{listings}
\newtheorem{lemma}[theorem]{Lemma}
\usepackage[left=3cm,right=3cm]{geometry}
\usepackage{bookmark}

\title{Generalized Werner's formula and the connection between trigonometry with target sum problems}
\author{Hayato Isa}
\date{June 2023}

\begin{document}

\maketitle
\begin{abstract}
This paper introduces the target sum function along with its characteristics. The target sum function takes a list of integers and a specific target integer as input values and expresses the number of ways to obtain the target sum by either adding or subtracting all of the integers from the given list. This function is rooted in the target sum problem which is explored in the fields of computer science and combinatorics.

This paper, following the establishment of the generalized Werner's formula, will propose a link between the target sum function and the definite integral of the product of sine and cosine functions based on the formula.
\end{abstract}
\providecommand{\keywords}[1]
{
  \small	
  \textbf{\textit{Keywords---}} #1
}
\keywords{target sum problem, Werner's formula, trigonometry}
\tableofcontents
\newpage

\section{Introduction}
\subsection{Target sum problem}
The target sum problem is a combinatorial problem that involves finding the number of ways to obtain a specific target integer by either adding or subtracting the elements from a given list of integers. The goal is to explore all possible combinations of additions and subtractions from the list to arrive at the target sum. For example, suppose we have the following list of integers: nums = [1, 2, 3] and the target integer: target = 0. 

Then, possible combinations to obtain the target sum of 5 using the numbers in the "nums" list are:
\begin{align*}
+1 + 2 - 3 = 0\\
-1 - 2 + 3 = 0\\
\end{align*}

So, in this example, there are two ways to obtain the target sum of 0.

The known most optimized way to solve this problem is by using dynamic programming and it allows us to solve the problem in pseudo-polynomial time $O(ns)$\\(cf. [1]).

In this paper, for an arbitrary sequence $X_n=\{b_i\in\mathbb{Z}\}_{i=0}^{n}$, $T(X_n)$ represents the solution of a target sum problem with a target integer $b_0$ and a list of integers [$b_1,b_2,...,b_n$]. Mathematically, $T(X_n)$ is the number of possible sign choices $\pm$ that satisfy the equation $ \pm b_1 \pm b_2 \pm \ldots \pm b_n = b_0$.
\subsection{notation and main results}
For an arbitrary sequence $X_n=\{b_i\in\mathbb{Z}\}_{i=0}^{n}$, define $T(X_n)$ as the number of possible sign choices $\pm$ that satisfy the equation $ \pm b_1 \pm b_2 \pm \ldots \pm b_n = b_0$.

Then the value of $T(X_n)$ can be computed as follows:
$$T(X_n)=\frac{2^{n}}{\pi}\int_{0}^{\pi}\prod_{i=0}^{n}\cos(b_i x)dx$$

For an arbitrary sequence $X_n=\{b_i\in\mathbb{Z}\}_{i=0}^{n}$, define $T_{e,m}(X_n),\quad -1\le m\le n$ as the number of possible sign choices $\pm$ that satisfy the equation $ \pm b_1 \pm b_2 \pm \ldots \pm b_n = b_0$ where the number of positive signs used for $b_1,b_2,...,b_m$ is even. Then, we can simply express the definite integral of the product of sin and cos by using $T(X_n),T_{e,m}(X_n)$.\\
For example, when $p+q\equiv 0\pmod{2},\quad p,q\in\mathbb{Z}$
\begin{equation*}
\int_{-q\pi}^{p\pi}\prod_{i=0}^{m}\sin(b_i x)\prod_{i=m+1}^{n}\cos(b_i x)\,dx=
\begin{cases}
0 & m\equiv 0\pmod{4}\\
\displaystyle-\frac{1}{2^n}(p+q)\bigg(2T_{e,m}(X_n)-T(X_n)\bigg) & m\equiv 1\pmod{4}\\
0 & m\equiv 2\pmod{4}\\
\displaystyle\frac{1}{2^n}(p+q)\bigg(2T_{e,m}(X_n)-T(X_n)\bigg) & m\equiv 3\pmod{4}\\
\end{cases}
\end{equation*}

\newpage
\section{Theorem and Proof}
\begin{theorem}
For all $X_n$, when we define $Y_n$ as a sequence whose 0th term and kth term are swapped, $T(X_n)=T(Y_n)$. $(0<k\le n)$
\end{theorem}
\begin{proof}
From the definition, $T(X_n)$ is the number of possible sign choices $\pm$ that satisfy the equation $\pm b_1 \pm b_2 \pm \ldots \pm b_n = b_0$. Which means $T(X_n)$ is the number of possible sign choices $\pm$ that satisfy the equation $b_k = b_0 \pm b_1 \pm b_2 \pm \ldots \pm b_{k-1} \pm b_{k+1} \pm \ldots \pm b_n$ or $-b_k = b_0 \pm b_1 \pm b_2 \pm \ldots \pm b_{k-1} \pm b_{k+1} \pm \ldots \pm b_n$. Which implies $T(X_n)$ is also the number of possible sign choices $\pm$ that satisfy the equation $b_k = \pm b_1 \pm b_2 \pm \ldots \pm b_{k-1} \pm b_0 \pm b_{k+1} \pm \ldots \pm b_n$
\end{proof}
\begin{lemma}
For all $X_n$
$$\int_{0}^{\pi}\sum_{\pm}\cos\left(\left(-b_0\pm b_1\pm...\pm b_n\right)x\right)\,dx=T(X_n)\pi$$
where the summation sums all $2^n$ possible $\cos\left(\left(-b_0\pm b_1\pm...\pm b_n\right)x\right)$
\end{lemma}
\begin{proof}
\begin{equation*}
\int_{0}^{\pi}\cos{nx}\,dx = \begin{cases}
\pi & \quad n=0 \\
0 & \quad n\neq0
\end{cases}
\end{equation*}
Therefore, $-b_0\pm b_2\pm...\pm b_n=0\iff \int_{0}^{\pi}\cos\left(\left(-b_0\pm b_1\pm...\pm b_n\right)x\right)\,dx=\pi$. Also, since $T(X_n)$ is the number of possible sign choices $\pm$ that satisfy the equation $ \pm b_1 \pm b_2 \pm \ldots \pm b_n = b_0$,
\begin{align*}
&\int_{0}^{\pi}\sum_{\pm}\cos\left(\left(-b_0\pm b_2\pm...\pm b_n\right)x\right)\,dx\\
&=\bigg[\text{the number of ways to choose $\pm$ to make }\pm b_1\pm b_2\pm...\pm b_n=b_0\bigg]\cdot\pi\\
&=T(X_n)\pi
\end{align*}
\end{proof}

\begin{lemma}
For all $X_n$
$$\prod_{i=0}^{n}\cos(b_i x)=
\frac{1}{2^{n}}\sum_{\pm}\cos\left(\left(-b_0\pm b_1\pm...\pm b_n\right)x\right)$$
where the summation sums all $2^n$ possible $\cos\left(\left(-b_0\pm b_1\pm...\pm b_n\right)x\right)$
\end{lemma}
\begin{proof}
When n=0, the equality holds.\\
When $n\geq0$, suppose
$$\prod_{i=0}^{n}\cos(b_i x)=
\frac{1}{2^{n}}\sum_{\pm}\cos\left(\left(-b_0\pm b_1\pm...\pm b_n\right)x\right)$$
Then,
\begin{align*}
&\prod_{i=0}^{n+1}\cos(b_i x)
=\cos(b_{n+1} x)\prod_{i=0}^{n}\cos(b_i x)=\frac{1}{2^{n}}\sum_{\pm}\cos\left(\left(-b_0\pm b_1\pm...\pm b_n\right)x\right)\cos(b_{n+1} x)\\
&=\frac{1}{2^{n+1}}\sum_{\pm}\bigg\{\cos\big(\left(-b_0\pm b_1\pm...\pm b_n\right)x-b_{n+1}x\big)+\cos\big(\left(-b_0\pm b_1\pm...\pm b_n\right)x+b_{n+1}x\big)\bigg\}\\
&=\frac{1}{2^{n+1}}\sum_{\pm}\cos\left(\left(-b_0\pm b_1\pm...\pm b_{n+1}\right)x\right)
\end{align*}
Therefore, equality holds for all $X_n$.
\end{proof}

\begin{lemma}
"Generalized werner's formula" For all $X_n$ and $-1\leq m\in\mathbb{Z}$
\begin{equation*}
\prod_{i=0}^{m}\sin(b_i x)\prod_{i=m+1}^{n}\cos(b_i x)=\frac{(-1)^{\lfloor{\frac{m+1}{2}}\rfloor}}{2^n}
\begin{cases}
\sum_{e\in S}\left(\cos\left(\left(b_0+e_1 b_1+...+ e_n b_n\right)x\right)\prod_{j=1}^{m}e_j\right) & \quad m\text{ is odd}\\
\sum_{e\in S}\left(\sin\left(\left(b_0+e_1 b_1+...+ e_n b_n\right)x\right)\prod_{j=1}^{m}e_j\right) & \quad m\text{ is even}
\end{cases}
\end{equation*}
where the summation will sum all $2^{n}$ possible $\cos\left(\left(b_0\pm b_1\pm...\pm b_n\right)x\right)$
\end{lemma}
\begin{proof}
From Lemma 3, we can say for all $X_n$
$$\prod_{i=0}^{n}\cos(b_i x)=
\frac{1}{2^{n}}\sum_{\pm}\cos\left(\left(b_0\pm b_1\pm...\pm b_n\right)x\right)$$
Considering applying $\prod_{i=0}^{m}\left(-\frac{\partial}{\partial b_i}\right)$ to both side of the equation\\
Then, when $m$ is odd number
$$\prod_{i=0}^{m}\sin(b_i x)\prod_{i=m+1}^{n}\cos(b_i x)=\frac{(-1)^{\lfloor{\frac{m+1}{2}}\rfloor}}{2^n}\sum_{e\in S}\left(\cos\left(\left(b_0+e_1 b_1+...+ e_n b_n\right)x\right)\prod_{j=1}^{m}e_j\right)$$
when $m$ is even number
$$\prod_{i=0}^{m}\sin(b_i x)\prod_{i=m+1}^{n}\cos(b_i x)=\frac{(-1)^{\lfloor{\frac{m+1}{2}}\rfloor}}{2^n}\sum_{e\in S}\left(\sin\left(\left(b_0+e_1 b_1+...+ e_n b_n\right)x\right)\prod_{j=1}^{m}e_j\right)$$
where $S=\{1,-1\}^n$\\
\end{proof}

\begin{theorem}For all $X_n$
$$T(X_n)=\frac{2^{n}}{\pi}\int_{0}^{\pi}\prod_{i=0}^{n}\cos(b_i x)dx$$
\end{theorem}
\begin{proof}
From Lemma 2, for all $X_n$
\begin{align*}
&\prod_{i=0}^{n}\cos(b_i x)=\frac{1}{2^{n}}\sum_{\pm}\cos\left(\left(-b_0\pm b_1\pm...\pm b_{n}\right)x\right)\\
\implies&\int_{0}^{\pi}\prod_{i=0}^{n}\cos(b_i x)dx=\frac{1}{2^{n}}\int_{0}^{\pi}\sum_{\pm}\cos\left(\left(-b_0\pm b_1\pm...\pm b_{n}\right)x\right)dx\\
\implies&\int_{0}^{\pi}\prod_{i=0}^{n}\cos(b_i x)dx=\frac{1}{2^{n}}T(X_n)\pi\quad\text{($\because$ Lemma 2)}
\end{align*}
$$\therefore T(X_n)=\frac{2^n}{\pi}\int_{0}^{\pi}\prod_{i=0}^{n}\cos(b_i x)dx$$
\end{proof}

\begin{theorem}For all $X_n$, when $p+q\equiv 0\pmod{2},\quad p,q\in\mathbb{Z}$
\begin{equation*}
\int_{-q\pi}^{p\pi}\prod_{i=0}^{m}\sin(b_i x)\prod_{i=m+1}^{n}\cos(b_i x)\,dx=
\begin{cases}
0 & m\equiv 0\pmod{4}\\
\displaystyle-\frac{1}{2^n}(p+q)\bigg(2T_{e,m}(X_n)-T(X_n)\bigg) & m\equiv 1\pmod{4}\\
0 & m\equiv 2\pmod{4}\\
\displaystyle\frac{1}{2^n}(p+q)\bigg(2T_{e,m}(X_n)-T(X_n)\bigg) & m\equiv 3\pmod{4}\\
\end{cases}
\end{equation*}
\end{theorem}
\begin{proof}
From Lemma 4, substituting $b_0=-b_0$,
$$\int_{-q\pi}^{p\pi}\prod_{i=0}^{m}\sin(b_i x)\prod_{i=m+1}^{n}\cos(b_i x)\, dx=$$
\begin{equation*}
\frac{(-1)^{\lfloor{\frac{m+1}{2}}\rfloor}}{2^n}
\begin{cases}
\int_{-q\pi}^{p\pi}-\sum_{e\in S}\left(\cos\left(\left(-b_0+e_1 b_1+...+ e_n b_n\right)x\right)\prod_{j=1}^{m}e_j\right)\, dx & \quad m\text{ is odd}\\
\int_{-q\pi}^{p\pi}-\sum_{e\in S}\left(\sin\left(\left(-b_0+e_1 b_1+...+ e_n b_n\right)x\right)\prod_{j=1}^{m}e_j\right)\, dx & \quad m\text{ is even}
\end{cases}
\end{equation*}
Considering the case when m is odd
\begin{equation*}
\int_{-q\pi}^{p\pi}\cos{nx}\,dx = \begin{cases}
(p+q)\pi & \quad n=0 \\
0 & \quad n\neq0
\end{cases}
\end{equation*}
Therefore, similar to lemma 2,
\begin{align*}
&\int_{-q\pi}^{p\pi}-\sum_{e\in S}\left(\cos\left(\left(-b_0+e_1 b_1+...+ e_n b_n\right)x\right)\prod_{j=1}^{m}e_j\right)\, dx\\
&=\int_{-q\pi}^{p\pi}-\sum_{e\in S}\left(\cos\left(\left(-b_0+e_1 b_1+...+ e_n b_n\right)x\right)\prod_{j=1}^{m}e_j\right)+\sum_{\pm}\cos\left(\left(-b_0\pm b_2\pm...\pm b_n\right)x\right)\, dx-(p+q)T(X_n)\pi\\
&=(p+q)\bigg[\text{two times the number of ways to choose $(e_1,\dots,e_n)$ to make } e_1 b_1+...+ e_n b_n=b_0,\, \prod_{j=1}^{m}e_j=-1\bigg]\pi\\
&\quad -(p+q)T(X_n)\pi\\
&=(p+q)\cdot 2T_{e,m}(X_n)\cdot\pi -(p+q)T(X_n)\pi
\end{align*}
$$\therefore\int_{-q\pi}^{p\pi}\prod_{i=0}^{m}\sin(b_i x)\prod_{i=m+1}^{n}\cos(b_i x)\, dx=\frac{(-1)^{\lfloor{\frac{m+1}{2}}\rfloor}}{2^n}(p+q)\bigg(2T_{e,m}(X_n)-T(X_n)\bigg)$$\\
Considering the case when m is even
\begin{equation*}
\int_{-q\pi}^{p\pi}\sin{nx}\,dx =\frac{1}{n}(\cos{nq\pi}-\cos{np\pi})
\end{equation*}
When $p+q\equiv 0 \pmod{2}$, $(p,q)=(2k,2l),(2k+1,2l+1)$, where $k,l\in\mathbb{Z}$\\
For both cases, $\cos{nq\pi}-\cos{np\pi}=0$, so $$\int_{-q\pi}^{p\pi}\sin{nx}\,dx = 0$$
$$\therefore\int_{-q\pi}^{p\pi}\prod_{i=0}^{m}\sin(b_i x)\prod_{i=m+1}^{n}\cos(b_i x)\, dx=0$$
\end{proof}

\end{document}